\documentclass[11pt]{amsart}
\usepackage{xcolor,cancel}
\usepackage[normalem]{ulem}
\usepackage{mathtools}

\usepackage[all]{xy}
\makeatletter
\def\resetMathstrut@{%
  \setbox\z@\hbox{%
    \mathchardef\@tempa\mathcode`\(\relax
    \def\@tempb##1"##2##3{\the\textfont"##3\char"}%
    \expandafter\@tempb\meaning\@tempa \relax
  }%
  \ht\Mathstrutbox@1.2\ht\z@ \dp\Mathstrutbox@1.2\dp\z@
}
\makeatother

\usepackage{pb-diagram}
\usepackage{graphicx}

\usepackage{comment}
\usepackage{stmaryrd}   % for the Kulkarni--Nomizu product

\usepackage{fullpage}
\usepackage[latin1]{inputenc}
\usepackage{hyperref}   % must be loaded before enumitem
\usepackage{enumitem}

\theoremstyle{definition}

\newtheorem{lemmat}{Lemma}[section]
\newtheorem{remark}[lemmat]{Remark}

\newtheorem{defn}[lemmat]{Definition}

\theoremstyle{plain}
\newtheorem{theorem}[lemmat]{Theorem}
\newtheorem{lemma}[lemmat]{Lemma}

\newtheorem{corollary}[lemmat]{Corollary}

\newtheorem{MainThm}{Theorem}

\newcommand{\Z}{\mathbb{Z}}           % integers
\newcommand{\Q}{\mathbb{Q}}

\newcommand{\Fr}{\mathrm{Fr}}

\newcommand{\id}{\mathrm{id}}

\newcommand{\Riem}{{\mathcal R}}
\newcommand{\Hom}{\mathrm{Hom}}

\newcommand{\KO}{\mathrm{KO}}
\newcommand{\Emb}{\mathrm{Emb}}
\newcommand{\Imm}{\mathrm{Imm}}
\newcommand{\ind}{\mathrm{ind}}
\newcommand{\Diff}{\mathrm{Diff}}

\newcommand{\BDiff}{\mathrm{BDiff}}
\newcommand{\Tor}{\mathfrak{Tor}}
\DeclareMathOperator{\psc}{\mathrm{psc}}
\DeclareMathOperator{\pRc}{\mathrm{pRc}}

\newcommand{\bQ}{\mathbb{Q}}

\newcommand{\trg}{\mathrm{trg}}
\newcommand{\inddiff}{\mathrm{inddiff}}

\newcommand{\ev}{\mathrm{ev}}
\newcommand{\Spin}{\mathrm{spin}}
\newcommand{\bZ}{\mathbb{Z}}

\newcommand{\ceil}[1]{\lceil #1 \rceil}
\newcommand{\spann}{\mathrm{span}}
\newcommand{\bN}{\mathbb{N}}

\newcommand{\ch}{\mathrm{ch}}

\begin{document}

\title{On the topology of the space of Ricci-positive metrics}
\author{Boris Botvinnik}
\address{
Department of Mathematics\\
University of Oregon \\
Eugene, OR, 97403\\
USA
}
\email{botvinn@uoregon.edu}
\author{Johannes Ebert}
\address{
Mathematisches Institut der Westf{\"a}lischen Wilhelms-Universit{\"a}t M{\"u}nster\\
Einsteinstr. 62\\
DE-48149 M{\"u}nster\\
Germany
}
\email{johannes.ebert@uni-muenster.de}
\author{David J.~Wraith}
\address{Department of Mathematics and Statistics\\
National University of Ireland Maynooth\\
Maynooth\\
Ireland}
\email{david.wraith@mu.ie}
\subjclass[2000]{53C27, 57R65, 58J05, 58J50}
\begin{abstract}
We show that the space $\Riem^{\pRc}(W_g^{2n})$ of metrics with
positive Ricci curvature on the manifold $W^{2n}_g := \sharp^g (S^n
\times S^n)$ has nontrivial rational homology if $n \not \equiv 3
\pmod 4$ and $g$ are both sufficiently large. The same argument
applies to $\Riem^{\pRc}(W_g^{2n} \sharp N)$ provided that $N$ is spin
and $W_g^{2n} \sharp N$ admits a Ricci positive metric.
\end{abstract}  
\maketitle
\vspace*{-10mm}

\section{Introduction}
Let $M$ be a smooth closed manifold of dimension $d$. We denote by
$\Riem(M)$ the space of all Riemannian metrics on $M$ with the
$C^\infty$-topology, and we let $\Riem^{\pRc}(M)\subset
\Riem^{\psc}(M) \subset \Riem(M)$ be the subspaces of metrics with
positive Ricci curvature and with positive scalar curvature
respectively (abbreviated later as ``pRc metrics'' and ``psc
metrics'').

In recent years, there has been a considerable effort to study the
topology of the space of all positive scalar curvature metrics on
manifolds, see \cite{HSS,BER-W,CSS,CS,P17-2}. The results achieved so
far show that roughly the space $\Riem^{\psc}(M)$ at least is as
complicated as real $K$-theory, provided one focuses on spin manifolds
in dimensions at least five.

The aim of this paper is to study the space of Ricci positive metrics
on certain manifolds. Very little is known about the topology of this
space. It was shown in \cite{Wraith} that the space of Ricci positive
metrics on all homotopy spheres which bound parallelisable manifolds
in dimensions $4n-1$ ($n \ge 2$) has infinitely many
path-components. More recently, it was shown in \cite{CSS} (discussed
below) that for spin manifolds in dimension at least six, there are
order two elements in certain higher homotopy groups of this space of
metrics. We will supplement these results by showing that the space of
Ricci positive metrics for certain manifolds can display non-trivial
higher rational cohomology.

If $M$ is a \emph{spin} manifold, there is a secondary index map
$\inddiff_g:\Riem^{\psc}(M) \to \Omega^{\infty+d+1}\KO$, depending on
the choice of a basepoint $g \in \Riem^{\psc}(M)$ and mapping into a
suitable space of the real $K$-theory spectrum. The map is constructed
using the spin Dirac operator and the Schr\"odinger-Lichnerowicz
formula. We refer to the introduction of \cite{BER-W} for an informal
discussion and to \cite[Section 3.3]{BER-W} for precise
definitions. On homotopy groups, the map $\inddiff_g$ induces a
homomorphism
\begin{equation}\label{eq:inddiff-on-homotopy}
(\inddiff_g)_*:  \ \pi_m (\Riem^{\psc}(M),g) \to KO_{d+m+1}:= KO_{d+m+1}(*),
\end{equation}
and it has been proven \cite{BER-W} that
\eqref{eq:inddiff-on-homotopy} is surjective after rationalization for
each spin manifold of dimension $d\geq 6$ that admits a psc
metric. More recently, Perlmutter \cite{P17-1,P17-2} showed how to
cover the case $d=5$. The proof of these results relies heavily on a
remarkable property of the space $\Riem^{\psc}(M)$: a strengthening of
the classical Gromov-Lawson surgery theorem \cite{GL} (see
\cite{Chernysh} and also \cite{WalshC}) shows that the homotopy type
of $\Riem^{\psc}(M)$ only depends on the spin cobordism class of $M$
(for simply connected $M$ of dimension $d \geq 5$).

Positivity of the Ricci curvature is a much more restrictive condition
than positivity of scalar curvature, and it is well-known that no
variant of the Gromov-Lawson surgery theorem can hold for positive
Ricci curvature, as any such variant would imply that $S^1 \times
S^{d-1}$ has positive Ricci curvature for large enough $d$,
contradicting Myers' theorem \cite[Theorem 25 on p. 171]{Petersen}. In
particular, the general approach of the paper \cite{BER-W} showing the
nontriviality of \eqref{eq:inddiff-on-homotopy} does not work when the
space $\Riem^{\psc}(M)$ is replaced by $\Riem^{\pRc}(M)$.

Another, older, approach to construct elements in $\Riem^{\psc}(M)$
which was pioneered by Hitchin \cite{HitchinSpin} has the advantage
that it does not require elaborate constructions of metrics of
positive scalar curvature, and clearly carries over to positive Ricci
curvature. One uses the action of the diffeomorphism group $\Diff (M)$
or better the \emph{spin diffeomorphism group}\footnote{\ Let us briefly
  recall the definition. Fix a spin structure on $s$ on $M$ and let
  $\widetilde{\Diff^{\Spin}} (M) \subset \Diff(M)$ be the (finite
  index open) subgroup of diffeomorphisms which fix $s$ up to
  isomorphism. Then $\Diff^{\Spin}(M)$ consists of all $(f,\hat{f})$,
  $f \in \widetilde{\Diff^{\Spin}} (M)$ and $\hat{f}$ is an
  isomorphism $f^* s \cong s$ of spin structures. The group
  $\Diff^{\Spin}(M)$ is a central extension of
  $\widetilde{\Diff^{\Spin}} (M)$ by $\bZ/2$.} $\Diff^{\Spin} (M)$ on
$\Riem^{\psc}(M)$.

For a fixed $g_0 \in \Riem^{\psc} (M)$, there is the orbit map
$\ev_{g_0}: \Diff (M)^{\Spin} \xrightarrow[]{f \mapsto f^* g_0}
\Riem^{\psc} (M)$, and we obtain the composition
\begin{equation}\label{eq:compositionasd}
  \pi_m (\Diff (M)^{\Spin}) \xrightarrow[]{(\ev_{g_0})_*}
  \pi_m (\Riem^{\psc} (M)) \xrightarrow[]{(\inddiff_{g_0})_*} KO_{d+m+1}.
\end{equation}
The homomorphism \eqref{eq:compositionasd} does not depend on the
choice of $g_0$ (i.e., the two dependencies on $g_0$ cancel out) and
can be computed by topological means, using the family index
theorem. Clearly, the homomorphism \eqref{eq:compositionasd} factors
through $\pi_m (\Riem^{\pRc}(M))$ as well, provided the fixed metric
$g_0$ has positive Ricci curvature. Crowley-Schick-Steimle
\cite[Corollary 1.9]{CSS} proved that if $d \geq 6$, there are
elements in $\pi_{8r+i-d}(\Diff^{\Spin}(M))$, $i=0,1$, which under the
homomorphism \eqref{eq:compositionasd} hit the nontrivial element of
$KO_{8r+i+1}=\bZ/2$, generalizing previous work by Hitchin
\cite{HitchinSpin} and Crowley-Schick \cite{CS}.

If $m=4k-d-1$, the target of the homomorphism
\eqref{eq:compositionasd} is $KO_{d+m+1}=\bZ$. Then the homomorphism
\eqref{eq:compositionasd} can be computed by means of characteristic
classes using the following recipe. Let $f:S^{4k-d-1} \to
\Diff^{\Spin}(M)$ be a map representing an element $[f] \in
\pi_{4k-d-1}(\Diff^{\Spin}(M))$. Since $\pi_{4k-d-1}(\Diff^{\Spin}(M))
\cong \pi_{4k-d}(B\Diff^{\Spin}(M)),$ we obtain a corresponding element
in $\pi_{4k-d}(B\Diff^{\Spin}(M))$ represented by a map
$\bar{f}:S^{4k-d} \to B\Diff^{\Spin}(M),$ with $\bar{f}$ uniquely
determined by $f$ up to homotopy. Using $\bar{f}$ to pull back the
universal spin-$M$-bundle\footnote{\  A \emph{spin-$M$-bundle} is a fibre
  bundle with fibre $M$ and structure group
  $\Diff^{\Spin}(M)$. Equivalently, it is a smooth fibre bundle with
  fibre $M$, together with a spin structure on its vertical tangent
  bundle.}, we form the spin-$M$-bundle $\pi:E_f \to S^{4k-d}$ with
vertical tangent bundle $T_v E_f$. Let $\hat{A}_k (T_v E_f) \in H^{4k}(E_f;\bQ)$ be the degree $4k$ component of the $\hat{A}$-class of $T_v E_f$. From this we can compute the
tautological characteristic class
\begin{equation*}
\kappa_{\hat{A}_k}(E_f):= \pi_! (\hat{A}_k (T_v E_f)) \in H^{4k-d} (S^{4k-d};\bQ) 
\end{equation*}
and integrate over $S^{4k-d}$ to obtain the \emph{$\hat{A}$-genus
  $\hat{A}(f) \in \bQ$} of the homotopy class of $f$. The
Atiyah-Singer family index theorem implies that $\hat{A}(f)$ is an
integer and that
\begin{equation*}
(\inddiff_g)_*(\ev_g)_*([f]) = a_k \hat{A}(f) \in KO_{4k}  \cong \bZ
\end{equation*}
where $a_k = 1$ if $k$ is even and $a_k =\frac{1}{2}$ if $k$ is
odd. We conclude: if $M$ admits a Ricci positive metric $g_0$ and if
there is $f \in \pi_{d+1-4k} (\Diff^{\Spin} (M))$ with $\hat{A}(f)\neq
0$, then $\pi_{d+1-4k}(\Riem^{\pRc}(M)) \otimes \bQ \neq 0$ (and
$(\inddiff_{g_0})_* \neq 0$ in that degree).

However, it seems to be a very difficult problem to construct such
elements in the homotopy groups of $\Diff^{\Spin} (M)$ when the
relevant $KO$-group is $\bZ$. (It is known that such manifolds $M$
exist, by \cite[Theorem 1.3]{HSS}, but the proof in that paper does
not provide a concrete construction of $M$, let alone an $M$ which has
a Ricci positive metric.)

In the present paper, we study the manifolds $W^{2n}_g := (S^n\times
S^n)^{\# g}$, the connected sum of $g$ copies of $S^n \times S^n$. It
has been proven by Sha and Yang \cite[Theorem 1]{ShaYang} that these
admit a metric of positive Ricci curvature, if $n \geq 2$. It is known
that for each $f \in \pi_* (\Diff^{\Spin}(W_g^{2n}))$, the
$\hat{A}$-genus is trivial, i.e. $\hat{A}(f)=0$. This follows from
\cite[Proposition 1.9]{HSS} since the Pontrjagin classes of $W^{2n}_g$
are all trivial. Hence with the method just described, one cannot
detect nontrivial elements in $\pi_*(\Riem^{\pRc}(W^{2n}_g))\otimes \bQ$. What we do in this paper is to
prove that the rational \emph{cohomology}
$H^j(\Riem^{\pRc}(W^{2n}_g);\bQ)$ is nontrivial, for some values of
$g$, $n$ and $j$.
\begin{defn}
Let $M$ be a closed manifold of dimension $2n$. The \emph{genus}
$g(M)$ of $M$ is the largest number $g$ so that $M$ can be written as
a connected sum $N \sharp W_g^{2n}$ for some smooth manifold $N$.
\end{defn}
Let us state our main result.
\begin{MainThm}\label{mainthm}
Let $M^{2n}$ be a closed spin manifold of genus $g$ and assume that $n
\not \equiv 3 \pmod 4$. Assume that $M$ admits a metric of positive
Ricci curvature. Then
\begin{enumerate}
\item[{\rm (1)}] If $n=4q \geq 14$ and $g \geq 21$, then $H^j
  (\Riem^{\pRc}(M);\bQ)\neq 0$ for some $1 \leq j \leq 7$.%\in \{1,\ldots,7\} $.
\item[{\rm (2)}] If $n=4q+1 \geq 11$ and $g \geq 17$, then $H^j
  (\Riem^{\pRc}(M);\bQ)\neq 0$ for some $1 \leq j \leq5$.%\in \{1,\ldots,5\} $.
\item[{\rm (3)}] If $n=4q+2 \geq 8$ and $g \geq 13$, then $H^j
  (\Riem^{\pRc}(M);\bQ)\neq 0$ for some $1 \leq j \leq3$.%\in \{1,2,3\}$.
\end{enumerate}
\end{MainThm}
\begin{comment}
\begin{MainThm}\label{mainthm}
Let $M^{2n}$ be a closed spin manifold of genus $g$ and assume that $n
\not \equiv 3 \pmod 4$. Assume that $M$ admits a metric of positive
Ricci curvature. Then
\begin{enumerate}
\item[{\rm (1)}] If $n=4q \geq 8$ and $g \geq 13$, then $H^j
  (\Riem^{\pRc}(M);\bQ)\neq 0$ for some $j \in \{1,2,3\} $.
\item[{\rm (2)}] If $n=4q+1 \geq 13$ and $g \geq 17$, then $H^j
  (\Riem^{\pRc}(M);\bQ)\neq 0$ for some $j \in \{1,2,3,4,5\}$.
\item[{\rm (3)}] If $n=4q+2 \geq 10$ and $g \geq 13$, then $H^j
  (\Riem^{\pRc}(M);\bQ)\neq 0$ for some $j \in \{1,2,3\}$.
\end{enumerate}
\end{MainThm}
\end{comment}
Even though we use the index theory of the Dirac operator, our proof
does \emph{not} yield the nontriviality of $(\inddiff_{g_0})_*$ on
homotopy or homology groups.%, unless $n \equiv 1\pmod 4$.

Using the known constructions of Ricci positive metrics on manifolds
of large genus, we obtain the following list of examples to which
Theorem \ref{mainthm} applies. Note that all examples admit spin
structures and are shown to have Ricci positive metrics in the sources
quoted.
\begin{corollary}\label{extra}
The conclusion of Theorem \ref{mainthm} holds for the manifold $M
\sharp W^{2n}_g$ (for $g$ and $n$ as in Theorem \ref{mainthm}) if $M$ is any
manifold from the following list.
\begin{enumerate}
\item[{\rm (1)}] $M=S^{2n}$ \cite[Theorem 1]{ShaYang};
\item[{\rm (2)}] $M$ is a connected sum with summands of the form ${\mathbb C}P^{d/2}\times S^{\mathbf{d}_I},$ ${\mathbb H}P^{d/4} \times S^{\mathbf{d}_I},$ ${\mathbb O}P^2 \times S^{\mathbf{d}_I},$ $S^d \times S^{\mathbf{d}_I},$
where $S^{\mathbf{d}_I}$ denotes the product of spheres $S^{d_1}\times \cdots \times S^{d_s}$ with $d_i\geq 3,$ and $d \ge 2.$ The quantities $d$ and $\mathbf{d}_I$ are allowed to vary from summand to summand, $d$ must be a multiple of four in the second case, and an odd multiple of two in the first case (to ensure that $\mathbb{C}P^{d/2}$ is spin). \cite[Theorem A]{Burdick2}, but see also \cite[Theorem B]{Wraith2} and \cite[Corollary 4.10]{Burdick}.
%$M=(S^{q_1} \times S^{r_1}) \sharp (S^{q_2} \times S^{r_2}) \sharp \cdots \sharp (S^{q_m} \times S^{r_m}),$ where %$q_i+r_i=2n$ and $q_i,r_i \ge 3$ for each $i$ \cite[Theorem A]{Wraith2};
\item[{\rm (3)}] $M=N \sharp (S^{q_1} \times S^{r_1}) \sharp \cdots
  \sharp (S^{q_m} \times S^{r_m}),$ where $q_i+r_i=2n,$ $q_i,r_i \ge
  3$ for each $i$, and $N$ is an $S^a$-bundle over a Ricci positive
  base space $X^b$ with compact Lie structure group, $a+b=2n,$ and
  $a>b \ge 3$ \cite[Theorem B]{Wraith2};
\item[{\rm (4)}] $M=N \sharp E_1 \sharp \cdots \sharp E_m,$ with $N$
  as in {\rm (3)} but satisfying the tighter dimensional restriction
  $a \ge b+5,$ and where $E_1,\cdots,E_m$ are linear sphere bundles
  with base $S^{c_i},$ fibre $S^{2n-{c_i}},$ and $a-1 \ge c_i \ge b+4$
  \cite[Theorem C]{Wraith2};
%\item[{\rm (5)}] for any $j,k,l \ge 0,$ $M=\sharp_j ({\mathbb
%  C}P^{n/2} \times S^n)$ or $M=(\sharp_j ({\mathbb C}P^{n/2} \times
%  S^n))\sharp(\sharp_k ({\mathbb H}P^{n/4} \times S^n))$ (both if
%  $n/2$ is odd, to ensure that $\mathbb{C}P^{n/2}$ is spin), or
%  $M=\sharp_k ({\mathbb H}P^4 \times S^{16}))\sharp(\sharp_l ({\mathbb
%    O}P^2 \times S^{16}))$ \cite[Corollary 4.10]{Burdick}.
%\mnote{je:
%it is $CP^{2n+1}$ that is spin.}
\end{enumerate}
\end{corollary}
%\begin{remark}
%\textcolor{blue}{Recently Burdick, see \cite{Burdick2},} 
%\textcolor{red}{established the existence of pRc metrics on a new infinite family of examples. To state this result, let %$S^{\mathbf{d}_I}$ denote the product of spheres $S^{d_1}\times \cdots \times S^{d_s}$, with $d_i\geq 3.$ Then for $d\ge %2$,} \textcolor{blue}{any manifold constructed as a connected sum} \textcolor{red}{with summands of the form} 
%\textcolor{blue}{
%\begin{equation*}  
%{\mathbb C}P^{d/2}\times S^{\mathbf{d}_I}, \ \ \  {\mathbb H}P^{d/4} \times
%S^{\mathbf{d}_I}, \ \ \  {\mathbb O}P^2 \times S^{\mathbf{d}_I}, \ \ \  S^d \times S^{\mathbf{d}_I},
%\end{equation*}
%}
%\textcolor{red}{\hspace{-0.32cm} admits a pRc metric. Note that the quantities $d$ and $\mathbf{d}_I$ are allowed to vary from summand to summand, and obviously $d$ must be even respectively a multiple of 4 in the first two cases.}
%\textcolor{red}{These connected sums} \textcolor{blue}{(for an appropriate choice of $d$ and $\mathbf{d}_I$)}
%\textcolor{red}{then provide infinitely many} \textcolor{blue}{new examples of manifolds $M$} \textcolor{red}{to which %Corollary \ref{extra} applies.} 
%\end{remark}
\subsection{Acknowledgments}
We would like to thank Achim Krause for pointing out a counterexample to an erronous argument in the first draft of this paper, and Manuel Krannich for spotting a mistake in the first preprint version. Furthermore, we thank the referee for suggesting a simpler exposition at various places in this paper. 
\section{Proofs}
\subsection{The detection theorem}
The proof of Theorem \ref{mainthm} is based on the following general
detection principle.
\begin{theorem}\label{thm:generaldetection}
Let $M$ be a $d$-dimensional closed smooth spin manifold. Assume that 
\begin{enumerate}
\item[{\rm(1)}] $M$ has a metric of positive Ricci curvature,
\end{enumerate}
and that there is a smooth
$M$-bundle $\pi:E \to B$ with a spin structure on its vertical tangent
bundle $T_v E$, such that
\begin{enumerate}
\item[{\rm (2)}] the base space $B$ is $0$-connected and has finite
  fundamental group, and
\item[{\rm (3)}] for some $k> d/4$, the tautological class
  $\kappa_{\hat{A}_k} (E) \in H^{4k-d}(B;\bQ)$ associated to the
  $k$-th component of the $\hat{A}$-class is nonzero.
\end{enumerate}
Then the cohomology group $H^j
(\Riem^{\pRc}(M);\bQ)$ is nontrivial for some $1 \leq j \leq 4k-d-1$.
\end{theorem}

\begin{remark}
The proof applies verbatim to any other curvature condition which implies positivity of scalar curvature, e.g. positive sectional curvature. We show below that a manifold as in Theorem \ref{mainthm} satisfies the hypotheses (2) and (3). On the other hand, manifolds of large genus do not admit metrics of positive sectional curvature, by Gromov's Betti number estimate \cite[Theorem 0.2.A]{GromovBetti}. More generally, we are not aware of any example of a manifold $M$ which carries a metric of positive sectional curvature and satisfies the hypotheses (2) and (3) of Theorem \ref{thm:generaldetection}.
\end{remark}

\begin{proof}[Proof of Theorem \ref{thm:generaldetection}]
Consider the universal cover $p: \tilde B \to B$. Since $\pi_1(B)$ is
a finite group, it follows from a transfer argument \cite[Proposition
  3G.1]{Hatcher} that the induced homomorphism
\begin{equation*}
  p^* : H^*(B;\Q) \to H^*(\tilde B;\Q)
\end{equation*}
is injective. In particular, this means that the pull-back
spin-$M$-bundle $p^*E \to \tilde B$ satisfies the hypotheses of the
theorem, and we can assume without loss of generality that the base
$B$ is simply-connected.

Associated with the fibre bundle $\pi:E \to B$, there is a fibration
\begin{equation}\label{detection-proof-finrations}
  T^{\pRc}(\pi) \xrightarrow[]{\Pi} B,
\end{equation}
whose fibre over $b \in B$ is the space
$\Riem^{\pRc}(\pi^{-1}(b))$. The construction of that fibration is as
follows: there is a $\Diff(M)$-principal bundle $Q \to B$ such that $Q
\times_{\Diff(M)} M \cong E$ (there is no need to take the spin
condition into account here). We define
\begin{equation*}
  T^{\pRc}(\pi):= Q  \times_{\Diff(M)} \Riem^{\pRc}(M).
\end{equation*}
The pullback $M$-bundle $\Pi^{*} E \to T^{\pRc}(\pi)$ admits a spin
structure since the original fibre bundle $\pi: E\to B$ does, and in
addition it has a fibrewise Riemannian metric of positive Ricci
curvature, and hence of positive scalar curvature. 

Now let $\ind_E \in KO^{-d} (B)$ be the index class of the fibrewise Dirac operator on $E$. The Schr\"odinger-Lichnerowicz formula implies that $\Pi^* \ind_E =0 \in KO^{-d}(T^{\pRc}(\pi))$. Under the composition 
\[
 KO^{-d}(B) \stackrel{c}{\to} K^{-d}(B) \stackrel{\ch_{2k }}{\to}  H^{4k-d}(B;\bQ)
\]
of the complexification map with the $2k$th component of the Chern character, the class $\ind_E$ is mapped to the tautological class $\kappa_{\hat{A}_k} (E)$ (this is the cohomological version of the Atiyah-Singer family index theorem for the spin Dirac operator). Therefore, we obtain that
\begin{equation}\label{eq:vanishingformula}
 \Pi^* \kappa_{\hat{A}_k} (E) =  \ch_{2k} (c(\Pi^* \ind_E)) = 0 \in H^{4k-d}(T^{\pRc}(\pi);\bQ).
\end{equation}
In other words: the nontrivial element $\kappa_{\hat{A}_k} (E) \in
H^{4k-d}(B;\bQ)$ lies in the kernel of $\Pi^*$. Now we consider the
Serre spectral sequence of the fibration
\eqref{detection-proof-finrations}. The homomorphism $\Pi^*:
H^{*}(B;\bQ) \to H^{*}( T^{\pRc}(\pi))$ agrees with the composition
\begin{equation*}
 \xymatrix{
   H^{4k-d}(B;\bQ) \ar[r]^-{\iota} & H^{4k-d}(B;H^0 (\Riem^{\pRc}(M);\bQ)) = 
   %\ar@{=}[r] &
   E_2^{4k-d,0} \ar@{->>}[r] &  E_{4k-d+1}^{4k-d,0}
   \ar@{>->}[r] & H^{4k-d}( T^{\pRc}(\pi)).
 }
\end{equation*}
The first homomorphism $\iota$ is injective: it is induced from the
injective (it is here that the hypothesis $\Riem^{\pRc}(M) \neq \emptyset$ is used) map $\bQ \to H^0 (\Riem^{\pRc}(M);\bQ)$ of coefficient
systems over $B$. Since the base $B$ is simply connected, an injective
map of rational coefficient systems is split-injective and therefore
induces an injective map in cohomology.

Then it follows that for some $2 \leq r \leq 4k-d$, the differential
$d_r: E_r^{4k-d-r, r-1} \to E_r^{4k-d,0}$ is nonzero. Hence
$E_r^{4k-d-r, r-1} \neq 0$, which forces $H^{4k-d-r} (B; H^{r-1}
(\Riem^{\pRc}(M);\bQ)) = E_2^{4k-d-r, r-1} \neq 0$. In particular
$H^{r-1} (\Riem^{\pRc}(M);\bQ) \neq 0$, as desired.
\end{proof}

\subsection{The Torelli group of $W_{g}^{2n}$ and its cohomology}
We now give the proof of Theorem \ref{mainthm}. It relies on
computations by the second-named author and Randal-Williams
\cite{ER-W} which by themselves rely on many deep ingredients from
high-dimensional manifold theory, such as \cite{GRWHomStabI} and
\cite{GRW14}, among other things.
\begin{defn}
Let $D\subset W_g^{2n}$ be a fixed embedded closed $2n$-disk, and let
$\Diff(W_g^{2n},D)$ denote the group of diffeomorphisms of $W_g^{2n}$
which restrict to the identity on some open neighborhood of the disk
$D$. The \emph{Torelli group} $\Tor_g^{2n}$ is the group of
diffeomorphisms $\phi \in \Diff(W_g^{2n},D)$ which induce the identity
in the middle homology group of $W_{k}^{2n}$:
\begin{equation*}
  \Tor_g^{2n} : = \{ \ \phi \ | \ \phi_*=\id : H_n(W_g^{2n};\Z) \to
  H_n(W_g^{2n};\Z) \ \}\subset \Diff(W_g^{2n},D).
\end{equation*}
\end{defn}
Clearly the inclusion $\Tor_g^{2n} \subset \Diff(W_g^{2n},D)$ induces
a map $\eta: B\Tor_g^{2n} \to \BDiff(W_g^{2n},D)$ of the corresponding
classifying spaces. There is a fibre sequence
\begin{equation}\label{fibre01}
B\Tor_g^{2n} \xrightarrow[ ]{\eta} \BDiff(W_g^{2n},D) \to
B\Gamma(W_g^{2n}),
\end{equation}
where $\Gamma(W_g^{2n})\subset \mathrm{GL}(H_n(W_g^{2n};\Z))$ is a
certain arithmetic group, see \cite[Section 2.1]{ER-W}.
\begin{lemma}\label{lem:torelli-finitelymanycomponents}
Let $n \geq 4$ and $n \not \equiv 3 \pmod 4$. Then $\pi_0
(\Tor_g^{2n})$ is a finite group.
\end{lemma}
\begin{proof}
This is a corollary of a result by Kreck \cite[Theorem
  2]{Kreck79}. Kreck studies the group of orientation-preserving
diffeomorphisms $\Diff^+(W^{2n}_g)$ (not just those preserving a disk)
and the subgroup of diffeomorphisms acting as the identity on $H_n
(W_g^{2n})$, which he denotes by $S \Diff(W_g^{2n})$. There is an
analogous fibre sequence $ B S\Diff(W^{2n}_g)\to B\Diff^+ (W^{2n}_g)
\to B \Gamma(W^{2n}_g)$, and the comparison with the fibre sequence
\eqref{fibre01} yields the following diagram of groups
\begin{equation}\label{lem:torelli-finitelymanycomponentsproof1}
 \xymatrix{
 0 \ar[r]&  \pi_0 (\Tor_g^{2n})\ar[r] \ar[d]&  \pi_0 (\Diff(W^{2n}_g,D)) \ar@{->>}[d] \ar[r] & \Gamma(W^{2n}_g) \ar @{=}[d] \ar[r] & 1\\
 0 \ar[r] & \pi_0 (S \Diff (W_{k}^{2n}))\ar[r] &  \pi_0 (\Diff^+ (W^{2n}_g)) \ar[r] & \Gamma(W^{2n}_g) \ar[r] & 1.
 }
\end{equation}
Kreck showed that the lower sequence is exact \cite[Theorem
  2]{Kreck79}, and the upper one
is exact as well, see \cite[Proposition 2.3]{ER-W}. It is part of the
proof of \cite[Proposition 2.3]{ER-W} that the middle vertical map is
surjective. Using the fibre sequence
\begin{equation*}
  \Diff (W_g^{2n},D) \to
  \Diff^+ (W_g^{2n}) \to \Emb^+ (D,W_{k}^{2n}),
\end{equation*}
the weak homotopy equivalence\footnote{This is a standard fact: the inclusion $ \Emb^+ (D,W_g^{2n})\to \Imm^+ (D,W_g^{2n})$ into the space of orientation-preserving immersions is a weak equivalence by a scaling argument, and the map $\Imm^+ (D,W_g^{2n}) \to \Fr^+ (W_{g}^{2n})$ that takes the derivative at the origin is a weak homotopy equivalence by the Smale--Hirsch immersion theorem.}
 $\Emb^+ (D,W_g^{2n}) \simeq \Fr^+(TW_g^{2n})$ of the space of orientation-preserving embeddings of $D$
to the oriented frame bundle of $W_g^{2n},$ and the fact that
$\Fr^+ (TW_g^{2n})$ is connected and has $\pi_1 (\Fr^+ (TW_g^{2n}))
\cong \bZ/2$, we deduce that the kernel of the middle vertical map in
\eqref{lem:torelli-finitelymanycomponentsproof1} has order $\leq 2$,
and hence
\begin{equation*}
 \pi_0 (\Tor_g^{2n}) \to \pi_0 (S \Diff (W_{k}^{2n}))
\end{equation*}
is surjective with finite kernel. Kreck \cite[Theorem 2]{Kreck79}
shows moreover the existence of an exact sequence $0 \to
\Theta_{2n+1}/\Sigma \to \pi_0 (S \Diff (W_{k}^{2n})) \to
\Hom(H_n(W_g^{2n};\bZ);S \pi_n (SO(n))) \to 0$, where $\Theta_{2n+1}$
is the group of homotopy spheres which is finite by \cite{KervMil} and
$\Sigma$ is a certain subgroup, and $S \pi_n (SO(n)) \subset \pi_n
(SO(n+1))$ is the image of the stabilization map $SO(n) \to
SO(n+1)$. The values of $S \pi_n (SO(n))$ are tabularized
in \cite[p. 644]{Kreck79}, and from that table, one reads off that
$\Hom(H_n(W_g^{2n};\bZ);S \pi_n (SO(n)))$ is finite unless $n \equiv 3
\pmod 4$.
\end{proof}
The hard part is now to show that the tautological classes
$\kappa_{\hat{A}_k} \in H^{4k-2n} (B \Tor_g^{2n};\bQ)$ are
nontrivial. This was essentially done in \cite{ER-W}, and we only
recall the main points. Consider the graded algebra
$\bQ[p_{\ceil{\frac{n+1}{4}}}, \ldots,p_n,e]/(e^2-p_n) $, which is
the rational cohomology algebra of the $n$-connected cover of
$BO(2n)$. Define a graded vector space by
\begin{equation*}
( \sigma^{-2n} \bQ[p_{\ceil{\frac{n+1}{4}}}, \ldots,p_n,e]/(e^2-p_n))_j:=
\begin{cases}
 (\bQ[p_{\ceil{\frac{n+1}{4}}}, \ldots,p_n,e]/(e^2-p_n))_{j+2n} & j >0, \\
 0 & j \leq 0.
\end{cases}
\end{equation*}
Assigning to $c \in \sigma^{-2n} \bQ[p_{\ceil{\frac{n+1}{4}}},
  \ldots,p_n,e]/(e^2-p_n)$ the tautological class $\kappa_c$ of the universal $W_{g}^{2n}$-bundle defines a
map of graded vector spaces
\begin{equation*}
 \sigma^{-2n} \bQ[p_{\ceil{\frac{n+1}{4}}}, \ldots,p_n,e]/(e^2-p_n) \to H^* (B \Diff (W_g^{2n});\bQ),
\end{equation*}
which extends to a map from the free graded commutative algebra on the
source. The main results of \cite{GRW14} and \cite{GRWHomStabI} show
that this map is an isomorphism in degrees $j \leq (g-5)/2$.

Now let $L_k \in \bQ[p_{\ceil{\frac{n+1}{4}}},
  \ldots,p_n,e]/(e^2-p_n)$ be the $k$th component of the Hirzebruch
$L$-class. It is a consequence of the Atiyah-Singer family index
theorem for the signature operator that $\kappa_{L_k} \in H^{4k-2n}(B
\Diff (W^{2n}_g);\bQ)$ maps to zero in $H^* (B \Tor^{2n}_{g};\bQ)$,
see \cite[Theorem 2.9]{ER-W}. We remark that in loc.cit., the version
$\mathcal{L}_k$ of the $L$-class introduced by Atiyah and Singer
\cite[p. 577]{ASIII} is used, which is more natural from the viewpoint
of the index theorem. The distinction is not important for us, since
the two classes are related by the formula $\mathcal{L}_k =
\frac{1}{2^{2k}} L_k$.
\begin{theorem}[Ebert, Randal-Williams \cite{ER-W}]\label{thm:erwtheorem}
Let $C_g^{2n}$ be the largest integer with
\begin{enumerate}
 \item $C_g^{2n} \leq (g-5)/2$,
 \item $C_g^{2n} \leq n-3$, 
 \item $\max \{ 2C_{g}^{2n}+7, 3C_g^{2n} +4\} \leq 2n$.
\end{enumerate}
Then the map 
\begin{equation}\label{erw-map}
V:=  \frac{ \sigma^{-2n} \bQ[p_{\ceil{\frac{n+1}{4}}}, \ldots,p_n,e]/(e^2-p_n)}{\spann \{ L_k \vert k \in \bN \}} \to H^* (B \Tor_g^{2n};\bQ)
\end{equation}
is injective in degrees $\ast \le C^{2n}_g$.
\end{theorem}
\begin{remark}
Theorem \ref{thm:erwtheorem} is not stated as such in \cite{ER-W}. It
is shown in \cite[\S 5]{ER-W} that a certain map $B \Tor_g^{2n} \to
\Omega^\infty F$ to a certain infinite loop space induces an
isomorphism $H^* (\Omega^\infty F;\bQ) \to H^* (B \Tor_g^{2n}
;\bQ)^{\Gamma (W_{g}^{2n})} $ in degrees $* \leq C^{2n}_g$. Hence a
fortiori, $H^* (\Omega^\infty F;\bQ) \to H^* (B \Tor_g^{2n} ;\bQ)$ is
injective. The rational cohomology of $\Omega^\infty F$ is the free graded-commutative $\bQ$-algebra
generated by the left-hand side of \eqref{erw-map}. This is discussed
in \cite[\S 5, \S 2.4]{ER-W}.
\end{remark}
\begin{remark}
Recently, Kupers and Randal--Williams \cite[Theorem A]{KRW}, \cite[Theorem A]{KRWa} gave a
complete computation of $H^* (B \Tor_g^{2n};\bQ)$, valid for all $n
\geq 3$ and large $g$. Even though it is far from obvious from the results
stated in \cite{KRW}, their calculations imply that the map \eqref{erw-map} is injective if $\ast \leq (g-5)/2$ and $n \geq 3$. This allows
the lower bounds on $n$ in Theorem \ref{mainthm} to be replaced by $n \geq 3$.
\end{remark}

\begin{comment}
This is not stated as such in \cite{ER-W}, but the proofs in that
paper prove the stronger result that the map from the graded
commutative algebra generated by the left-hand side to the
$\Gamma_{g}^{2n}$-invariant part of the target is an isomorphism in
the indicated stable range. A fortiori, \eqref{erw-map} is injective.
\end{comment}

\begin{lemma}\label{lem:characetristtic-class-computation}
If $k \geq 2\ceil{\frac{n+1}{4}}$, the image of $\hat{A}_k$ in the
quotient space $V$ on the left hand side of \eqref{erw-map} is
nonzero.
\end{lemma}
\begin{proof}
This follows from a characteristic class computation that has been carried out in \cite[Section 5]{HSS}. If $k \geq 2\ceil{\frac{n+1}{4}}$, it can be written in the form $k=a+b$ with $\ceil{\frac{n+1}{2}} \leq a \leq b$. Consider the quotient $V'$ of $V$ by the subspace generated by all monomials \emph{except} $p_k$ and $p_a p_b$ (note that $\dim (V') \leq 2$). By \cite[Proposition 5.3]{HSS}, the image of $\hat{A}_k$ in $V'$ is nonzero and hence it must be nonzero in $V$. 
%To state the result, let
%$\ceil{\frac{n+1}{2}}\leq a, b $, $a+b = k$. Consider the quotient
%$V'$ of $V$ by the subspace generated by all monomials \emph{except}
%$p_n$ and $p_a p_b$ (the space $V'$ is at most $2$-dimensional). By
%\cite[Proposition 5.3]{HSS}, the image of $\hat{A}_k$ in $V'$ is
%nonzero and hence it must be nonzero in $V$. If $k \geq
%2\ceil{\frac{n+1}{4}}$, it can be written in the form $k=a+b$ with
%$\ceil{\frac{n+1}{2}} \leq a \leq b$.
\end{proof}
\begin{proof}[Proof of Theorem \ref{mainthm}]
Write $M = N \sharp W_{g}^{2n}$. The group $\Diff(W^{2n}_g;D)=
\Diff_\partial (W^{2n}_g\setminus D)$ and hence its subgroup $
\Tor_g^{2n}$ acts by diffeomorphisms on $M$: extend diffeomorphisms by
the identity over all of $M$.  Consider the fibre bundles
\[
 E_0:=  E \Tor_{g}^{2n} \times_{\Tor_{g}^{2n}} W^{2n}_g \to B \Tor_{g}^{2n} =:B
\]
and 
\[
 E:=  E \Tor_{g}^{2n} \times_{\Tor_{g}^{2n}} M \to B
\]
with fibres $W_{g}^{2n}$ and $M$, respectively.  An obstruction theory
argument (see e.g. \cite[Lemma 3.29]{BER-W}) shows that these bundles
admit spin structures (since the structure group fixes a disc in the fibre, which is connected).  By Lemma \ref{lem:torelli-finitelymanycomponents}, $B$ has
finite fundamental group. 

Depending on the value of $n$ modulo $4$, consider the smallest $k_0 \in \bN$ such that $k_0 \geq 2 \ceil{\frac{n+1}{4}}$; the class $\hat{A}_{k_0}$ is an element in the degree $q_0=4k_0-2n$ part of the space $V$. Then $q_0= 8,6,4$ in the cases $n \equiv 0,1,2 \pmod 4$. For $n$ and $g$ obeying the bounds stated in Theorem \ref{mainthm}, it follows from Theorem \ref{thm:erwtheorem} and Lemma \ref{lem:characetristtic-class-computation} that the tautological class $\kappa_{\hat{A}_{k_0}}(E_0)\in H^{q_0}(B;\bQ)$ is nontrivial. 

The bundle $E$ is fibrewise cobordant to the disjoint union of $E_0$ with the trivial bundle $B
\times N$. The classical argument \cite[Lemma 17.3]{MS} for the
cobordism invariance of Pontrjagin numbers shows that
\[
 \kappa_{\hat{A}_k} (E) =  \kappa_{\hat{A}_k} (E_0)+   \kappa_{\hat{A}_k} (B \times N)=  \kappa_{\hat{A}_k} (E_0)
\]
(since $4k-2n>0$, the second summand is zero). Hence
all hypotheses of Theorem \ref{thm:generaldetection} are satisfied,
and the proof is complete.
\end{proof}

\bibliographystyle{plain}
\bibliography{ricci}
\end{document}